\documentclass[9pt]{article}%

\usepackage{subfigure}
\usepackage{algorithm}
\usepackage{epsfig}
\usepackage{epstopdf}
\usepackage{graphics,graphicx,amssymb,amsmath,verbatim}
\usepackage{amssymb}
\usepackage{mathrsfs}
\usepackage{amsfonts}
\usepackage{amsmath}
\usepackage{graphicx}
\usepackage{color}
\usepackage{mathrsfs}
\usepackage{makecell}%
\providecommand{\U}[1]{\protect \rule{.1in}{.1in}}
\newtheorem{theorem}{Theorem}

\newtheorem{definition}{Definition}
\newtheorem{example}{Example}

\newtheorem{lemma}{Lemma}

\newtheorem{remark}{Remark}

\newenvironment{proof}[1][Proof]{\noindent \textbf{#1.} }{\  \rule{0.5em}{0.5em}}
\definecolor{blue}{rgb}{0,0,1}
\allowdisplaybreaks[4]
\begin{document}
\title{Exact Prescribed-Time Control Based on Simple Harmonic Motion: Stability Analysis and Nonsingular Sliding Mode Stabilization}

\author{Yi Ding and Bin Zhou,  \thanks{This work was supported
in part by the National Natural Science Foundation of China for Distinguished
Young Scholars under Grant 62125303,  the National Natural Science Foundation of China - ``Qisun Ye" Science Foundation under Grant U2441243, National Natural Science Foundation of China under Grant 62521005, Fundamental and Interdisciplinary Disciplines Breakthrough Plan of the Ministry of Education of China under Grant JYB2025XDXM206 and the Science Center Program of National
Natural Science Foundation of China under Grant 62188101.
\textit{(Corresponding author: Bin Zhou)}} \thanks{Yi Ding and Bin Zhou  are with the Center for Control Theory and Guidance
Technology, Harbin Institute of Technology, Harbin, 150001, China. (Email:
binzhoulee@163.com; binzhou@hit.edu.cn.)}}
\date{}
\maketitle

\begin{abstract}
This paper investigates the problems of stability analysis and nonsingular sliding mode stabilization for exact prescribed-time control. By exploiting the isochronism of simple harmonic motion, this paper establishes a novel exact prescribed-time control framework. First, a novel Lyapunov analysis method for exact prescribed-time control is proposed, based on which the exact prescribed-time stabilization of scalar systems is achieved. It is theoretically established that for arbitrary non-zero initial conditions, the settling time is exactly equal to the prescribed time. Next, the framework is extended to a novel sliding mode control law. It guarantees exact prescribed-time convergence for almost all initial conditions even in the presence of external disturbances. Notably, the proposed control law is nonsingular across the entire state space and maintains uniformly bounded gains. Finally, simulation results validate the effectiveness of the proposed methods.


\end{abstract}


\section{Introduction}
The earliest form of finite-time stability appeared in the 1960s under the name ``short-time stability" \cite{Dorato61}. The finite-time Lyapunov stability theory and finite-time homogeneity theory were formulated in \cite{Bhat97con} and \cite{Bhat00siam}, respectively. Owing to its potential for fast convergence, high precision, and improved robustness compared with asymptotic stability, finite-time stability has attracted sustained attention, leading to a wide range of results (see, for example, \cite{Cortes06auto,Efimov15tac,Franceschelli14tac,Franceschelli16tac,LiH23tac,Perruquetti08tac,Ortega20tac,Orlov05siam} and \cite{Zhao20tac} and the references therein).

Although the aforementioned works achieve finite-time stabilization for different types of systems, the corresponding settling time remains dependent on the initial conditions. This coupling leads to two issues: (i) the initial condition may be unavailable in advance, making it difficult to estimate the settling-time; and (ii) even if the initial state is known, the corresponding settling time can grow without bound as the initial condition moves farther away from the equilibrium.

Fixed-time stability, proposed in \cite{Z82tac}, addresses the above issues by guaranteeing a constant upper bound on the settling time that does not depend on initial conditions. Representative advances include homogeneity-based analysis/design \cite{Polyakov23,Tian17auto}, sliding-mode approaches \cite{Corradini18auto,Moulay21tac,Zuo15iet}, parameter estimation \cite{Wang20ejc} and event-triggered control \cite{Lei22tac}.  Nevertheless, while the upper bound of the settling time is initial-condition independent, the actual settling time of the closed-loop system may still vary with the initial condition.

Compared with finite-time and fixed-time control, prescribed-time control offers two notable advantages: (i) the actual settling time is independent of the initial conditions, and (ii) the settling time can be preassigned to an arbitrary positive constant. In \cite{Song17auto}, the concept of prescribed-time control was first established, and a time-varying high-gain feedback (THF) control scheme was designed to achieve the prescribed-time stabilization of single-input nonlinear systems. Based on THF, a large number of interesting and important works have been generated (see, for example, \cite{Gong20tnse,Holloway19auto,Li22tac,Li21tac,Li23siam,Kan17dmc,Krishnamurthy20auto,Ye25auto,Yucelen18tac,Zhou26auto,Zhang24auto}). However, a shared limitation of these methods is that the gain becomes unbounded, which introduces a singularity problem and makes the control law undefined at the prescribed time and thereafter. To address this issue, some interesting results were introduced in \cite{Orlov22auto}, providing bounded-gain methods. Nevertheless, for these bounded-gain methods, the actual settling time still varies with the initial conditions.

In addition to the THF, a novel prescribed-time control approach utilizing periodic gains and artificial delays, termed periodic delayed feedback (PDF), was originally proposed in \cite{zhou21tac}. By using uniformly bounded time-varying periodic gains and introducing an artificial time-delay term in the control law, the PDF can achieve a control performance similar to that of the THF. Building upon the ideas in \cite{zhou21tac}, several interesting results have emerged in recent years (see, for example, \cite{Ding23tac,Ding26tac} and \cite{Espitia26tac}). As the PDF is designed without relying to high gains, it is not subject to the singularity problems that typically arise in the THF framework. However, these results rely on time-varying periodic gains, which introduces extra offline computation.

SHM is one of the most fundamental and representative forms of oscillatory motion in classical mechanics. Research on SHM can be traced back to the early 17th century, when Galileo Galilei examined the isochronous behavior of pendulums \cite{Galilei14}. Leonhard Euler and Daniel Bernoulli utilized calculus to formulate simple harmonic motion, establishing it as a core component of analytical mechanics \cite{Euler}. Nowadays, the primary identity of SHM is no longer an ``object of study" in itself, but rather a fundamental unit derived from the local linearization or modal decomposition of numerous complex systems.

Particularly relevant to this paper, SHM exhibits isochronicity, meaning that the oscillation period is independent of the vibration amplitude. In ideal simple harmonic motion (e.g., a spring-mass system), the oscillation period $T$ is determined solely by the intrinsic properties of the system (such as mass $M$ and the spring constant $k$) and is independent of its initial conditions (a more comprehensive treatment is provided in Subsection \ref{intro_shm}). This isochronous property is remarkably similar to the control objectives of prescribed-time control, where the settling time remains independent of the initial conditions.

Motivated by the aforementioned observations, this paper aims to propose a novel exact prescribed-time control framework, outside of THF and PDF, by leveraging the isochronism of SHM. The main contributions of this paper are twofold:
\begin{enumerate}
\item This paper proposes a SHM based Lyapunov analysis method, and achieves exact prescribed-time stabilization for scalar systems.  It is guaranteed that for any  non-zero initial condition, the actual settling time of the closed-loop system exactly matches the prescribed time.

\item This paper proposes a novel exact prescribed-time nonsingular sliding mode control law, with the following appealing properties: (i)
    it guarantees that for almost all initial conditions, the actual settling time of the closed-loop system exactly matches the prescribed time, even under external disturbances, (ii) the proposed control law is clearly defined throughout the entire state space, and (iii) since the gains are uniformly bounded, the control law is well-defined for all $t\in [0,\infty)$.
\end{enumerate}

The remainder of this paper is organized as follows. Section II introduces some necessary preliminaries regarding stability definitions and the isochronicity of simple harmonic motion. Section III presents the core idea of this paper, establishing a novel SHM-like nonlinear system and a new Lyapunov-based analysis method, which is then applied to achieve exact prescribed-time control for scalar systems. In Section IV, an exact prescribed-time nonsingular sliding mode control law is proposed and its effectiveness is validated through a simulation example. Finally, Section V concludes the paper.

\section{Preliminary}
\subsection{Notation}

For $z\in \mathbf{R}$ and $\alpha\in (0,\infty)$, we denote $\mathrm{sig}^{\alpha}(z)=\vert z\vert^{\alpha}\mathrm{sign}(z)$, where $\mathrm{sign}(\cdot)$ denotes the sign function.  For $x\in \mathbf{R}$ and $y\in \mathbf{R}$, we denote
\begin{align*}
\mathrm{atan2}(y,x)=
\left\{
\begin{array}
[c]{ll}
\arctan\left(\frac{y}{x}\right), & x>0,\\
\arctan\left(\frac{y}{x}\right)+\pi, & x<0,y\geq 0,\\
\arctan\left(\frac{y}{x}\right)-\pi, & x<0,y<0,\\
\frac{\pi}{2}, & x=0,y>0,\\
-\frac{\pi}{2}, & x=0,y<0,\\
0, & x=0,y=0.
\end{array}
\right.
\end{align*}

\subsection{Finite-Time, Fixed-Time and Exact Fixed-Time Stability}
Consider the following nonlinear system:
\begin{equation}
\dot{x}(t)=f(t,x(t),\varepsilon(t)),\quad x(0)=x_{0},\quad t\geq 0,  \label{nonlinear_system}%
\end{equation}
where $x(t)\in \mathbf{R}^{n}$ is the state,
$f:[0,\infty)\times\mathbf{R}^{n}\times \mathbf{R}^{m}\rightarrow
\mathbf{R}^{n}$ is locally essentially bounded and locally
measurable, and the variable $\varepsilon(t)\in \mathbf{R}^{m}$ can be described as
\begin{equation}
\dot{\varepsilon}(t)=g(t,x(t)),\quad \varepsilon(0)=\varepsilon_{0},\quad t\geq 0, \label{non_sys2}
\end{equation}
where $g:[0,\infty)\times\mathbf{R}^{n}\rightarrow
\mathbf{R}^{m}$ is locally essentially bounded and locally
measurable.
\begin{remark}
Denote $\xi(t)=[x^{\mathrm{T}}(t),\varepsilon^{\mathrm{T}}(t)]^{\mathrm{T}}\in \mathbf{R}^{n+m}$, then (\ref{nonlinear_system}) and (\ref{non_sys2}) can be rewritten as
\begin{equation}
\dot{\xi}(t)=h(t,\xi(t))\triangleq \left[
\begin{array}
[c]{c}
f(t,x(t),\varepsilon(t))\\
g(x(t))
\end{array}
\right],\quad t\geq 0. \label{sys_sum}
\end{equation}
The function $\xi(t)$ defined on a non-degenerate interval $\mathbf{I} \subseteq [0,\infty)$ is called a Filippov solution \cite{Filippov88,Cortes08csm} for system (\ref{sys_sum}), if it is absolutely continuous on any subinterval $[t_{1},t_{2}]$ of $\mathbf{I}$, and for almost everywhere $t\in \mathbf{I}$,
\[
\dot{\xi}(t)\!\in\! H(t,\xi(t))\!\triangleq\! \bigcap_{\rho>0} \bigcap_{\mu(\mathbf{N})=0} \bar{\mathrm{co}}[h(t,\mathcal{B}(\xi(t),\rho)\backslash\mathbf{N})],
\]
where $\mu(\mathbf{N})$ is the Lebesgue measure of set $\mathbf{N}$, intersection is taken over all sets $\mathbf{N}$ of measure zero and over all $\rho>0$, $\mathcal{B}(\xi(t),\rho)$ is the ball with the center at $\xi(t)\in \mathbf{R}^{n+m}$ and the radius $\rho>0$, and $\bar{\mathrm{co}}[\mathbf{E}]$ is the closure of the convex hull of some set $\mathbf{E}$.
\end{remark}

It should be noted that all closed-loop systems discussed in this paper take the form of (\ref{sys_sum}). $x(t)$ is the state of primary concern, while $\varepsilon(t)$ appears as a controller parameter and serves as an extended state artificially introduced for control purposes. Therefore, this paper only concerns the stability of system (\ref{nonlinear_system}). Furthermore, as can be seen from (\ref{non_sys2}), for a given initial value $\varepsilon_{0}$, $\varepsilon(t)$ is actually determined only by the current and historical information of state $x(t)$.

Let the origin $x(t)=0$ be an equilibrium of system (\ref{nonlinear_system}). The system (\ref{nonlinear_system}) may have non-unique solutions and may admit both weak and strong stability (see \cite{Filippov88}). This paper deals only with the strong stability, which asks for stable behavior of all solutions of the system~(\ref{nonlinear_system}).

\begin{definition}\cite{Bhat00siam,Orlov05siam}
System (\ref{nonlinear_system}) is finite time stable (FTS) if it
is globally stable and there exists a function $T_{0}:\mathbf{R}^{n}\rightarrow [0,\infty)$, such that $x(t)=0,\, \forall t\geq T_{0}(x_{0})$.
\end{definition}
\begin{definition}
\label{def2} \cite{Z82tac} System (\ref{nonlinear_system}) is fixed-time stable (FxTS) with the settling time $T<\infty$ if it is FTS and $T_{0}(x_{0})\leq T$.
\end{definition}
\begin{definition}
\label{def3} \cite{Ding23tac} System (\ref{nonlinear_system}) is exact fixed-time stable (EFxTS) with the settling time $T<\infty$ if it is FxTS and $T_{0}(x_{0})= T$ for all $x_{0}\neq 0$.
\end{definition}

\subsection{Isochronicity of SHM}\label{intro_shm}
SHM is an oscillatory motion in which the restoring force (or restoring torque) is directly proportional to the displacement from equilibrium and acts in the opposite direction.

Mathematically, a motion is SHM if its displacement $x(t)\in \mathbf{R}$ satisfies
\begin{equation}
\ddot{x}(t)+\omega^2 x(t)=0,\quad x(0)=x_{0},\quad \dot{x}(0)=v_{0},\label{shm_sys}
\end{equation}
where $\omega>0$ is a constant. Let $x_{1}(t)= x(t)$ and $x_{2}(t)=-\dot{x}(t)/\omega$. Then the system can be written in the state-space form:
\begin{align}
\left[
\begin{array}
[c]{c}
\dot{x}_{1}(t)\\
\dot{x}_{2}(t)
\end{array}
\right]=\left[
\begin{array}
[c]{cc}
0 & -\omega\\
\omega & 0
\end{array}
\right]\left[
\begin{array}
[c]{c}
x_{1}(t)\\
x_{2}(t)
\end{array}
\right],\label{shm_state}
\end{align}
whose solution is
\begin{equation}
x_{1}(t)= r\cos(\omega t+\phi),\quad x_{2}(t)=r\sin(\omega t+\phi),\label{shm_x}
\end{equation}
where $
r=\sqrt{x_{1}^{2}(0)+x_{2}^{2}(0)}$ and $ \phi=\mathrm{atan2}\left(x_{2}(0),x_{1}(0)\right)$.

Next, we consider the special case where $v_{0}=0$. Then, (\ref{shm_x}) becomes
\[
x_{1}(t)=  x_{0}\cos(\omega t),\quad x_{2}(t)=x_{0}\sin(\omega t).
\]
Thus, for any $x_{0}\in \mathbf{R}$, it always holds that $x(\pi/(2\omega))=0$. In addition, for $x_{0}\neq 0$, we know that $x(t)\neq 0,\forall t\in [0,\pi/(2\omega))$. An example is provided below to demonstrate this:
\begin{example}
Consider a mass $M$ attached to a spring with spring constant $k$ on a frictionless surface (see Fig. \ref{shm}).
\begin{figure}[h]
\centering
\includegraphics[scale=0.9]{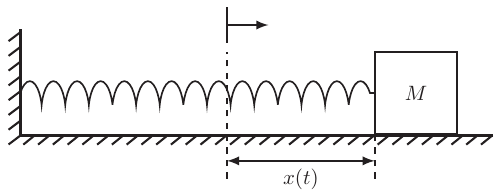}\caption{Ideal mass-spring oscillator}%
\label{shm}%
\end{figure}
Suppose the mass is pulled away from the equilibrium position and then released from rest at a displacement $x(0)=x_{0}$. It is straightforward to see that the force acting on the mass is $-kx(t)$. Therefore, the system dynamics can be written as:
\begin{equation}
\ddot{x}(t)+\frac{k}{M}x(t)=0,\quad x(0)=x_{0},\quad \dot{x}(0)=0. \label{shm_exp}
\end{equation}
It follows that
\[
x(t)=x_{0}\cos\left(\sqrt{\frac{k}{M}}t\right).
\]
For simulation, the spring constant $k$ is set to $k=1\mathrm{N/m}$ and the mass $M$ is set to $M=4/\pi^2\mathrm{kg}$. Fig. \ref{shm_sim} shows the evolution of the mass displacement under different initial conditions.
\begin{figure}[h]
\centering
\includegraphics[scale=0.9]{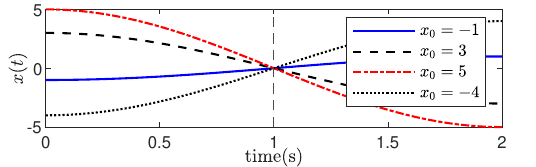}\caption{Displacement $x(t)$ of system (\ref{shm_exp}) under different initial values}%
\label{shm_sim}%
\end{figure}
\end{example}

From the simulation results, it can be seen that for different nonzero initial conditions, the displacement $x(t)$ always reaches $x(t)=0$ at the same time, which is very similar to the performance required by exact prescribed-time control. In fact, if we only consider $t\in [0,1]$, this is almost the same behavior as that targeted by exact prescribed-time control. Unfortunately, since $\dot{x}(t)\neq 0$ at $x(t)=0$, the state $x(t)$ cannot remain for $t\in [1,\infty)$. Hence, the system is not EFxTS. The remainder of this paper aims to remedy this limitation by proposing a novel exact prescribed-time control approach.

\section{Core Idea}
\subsection{A SHM-Like Nonlinear System}
Consider a nonlinear system in a SHM-like form:
\begin{align}
\left\{
\begin{aligned}
\dot{z}_{1}(t)= & -\frac{\omega}{\alpha}  z_{2}(t) \mathrm{sig}^{1-\alpha}(z_{1}(t)) ,\\
\dot{z}_{2}(t)= & \omega \vert z_{1}(t)\vert^{\alpha},
\end{aligned}
\right. \label{core}
\end{align}
where $\omega >0$ and $\alpha\in (0,1)$ are two constants.
\begin{lemma}\label{lem1}
The state of the system (\ref{core}) satisfies $z_{1}(t)=0,\forall t\geq \pi/\omega$. In addition, if $z_{1}(0)\neq 0$ and $z_{2}(0)=0$, then $z_{1}(t)=0,\forall t\geq \pi/(2\omega)$ and $z_{1}(t)\neq 0,\forall t\in [0,\pi/(2\omega))$.
\end{lemma}
\begin{proof}
We will prove it by contradiction. Assume that $z_{1}(t)\neq 0,\forall t\in [0,\pi/\omega]$. Define $\eta(t)=\vert z_{1}(t)\vert^{\alpha}$. It follows that
\begin{align}
\left\{
\begin{aligned}
\dot{\eta}(t)= & -\omega z_{2}(t),\\
\dot{z}_{2}(t)= & \omega \eta(t),
\end{aligned}
\right. \label{eta_lin}
\end{align}
for all $t\in[0,\pi/\omega]$. Note that system (\ref{eta_lin}) has the same form as system (\ref{shm_state}). Thus,
\begin{equation}
\eta(t)=r\cos(\omega t+\phi),\, z_{2}(t)=r\sin(\omega t+\phi),\, t\in[0,\pi/\omega], \label{lin_so}
\end{equation}
where $r=\sqrt{\eta^2(0)+z_{2}^2(0)}$ and $\phi=\mathrm{atan2}(z_{2}(0),\eta(0))$. There must exist a $t_{1}\in [0,\pi/\omega]$, such that $\omega t_1+\phi=\pi/2+k\pi$, where $k$ is an integer, which implies that $z_{1}(t_1)=\eta(t_1)=0$. This contradicts the assumption. Thus there exists $t_{\ast}\in [0,\pi/\omega]$ such that $z_{1}(t_{\ast})=0$. Note that for any constant $c$,  $(z_{1}(t),z_{2}(t))=(0,c)$ is an equilibrium of (\ref{core}). Then $z_{1}(t_{\ast})=0$ implies $z_{1}(t)=0,\forall t\geq t_{\ast}$, which further implies $z_{1}(t)=0,\forall t\geq \pi/\omega$.

Next, consider the case $z_{1}(0)\neq 0$ and $z_{2}(0)=0$. From the above, it is known that there exists $t_{\ast}\in [0,\pi/\omega]$ such that $z_{1}(t)=0,\forall t\geq t_{\ast}$ and $z_{1}(t)\neq 0,\forall t\in [0,t_{\ast})$. Thus, according to (\ref{eta_lin}) and (\ref{lin_so}), we have
\begin{equation*}
z_{1}(t)= \left\{
\begin{array}
[c]{ll}
 z_{1}(0)\vert \cos(\omega t)\vert^{1/\alpha}, & t\in [0,t_{\ast}),\\
0, & t\in [t_{\ast},\infty).
\end{array}
\right.
\end{equation*}
Since $z_{1}(t)$ is continuous, we know that $t_{\ast}=\pi/(2\omega)$. The proof is finished.
\end{proof}

\begin{remark}\label{rem2}
A feature of system (\ref{core}) is that, if $z_{2}(0)=0$, the settling time of the subsystem $\dot{z}_{1}(t)=-({\omega}/{\alpha}) z_{2}(t) \mathrm{sig}^{1-\alpha} (z_{1}(t))$ is exactly $\pi/(2\omega)$ for any non-zero initial condition $z_{1}(0)$.
\end{remark}

\begin{remark}
From the proof of Lemma \ref{lem1}, it can be seen that system (\ref{core}) is in fact obtained from the SHM system (\ref{shm_state}) via a nonlinear transformation. The main advantage of system (\ref{core}) is that $z_{1}(t)=0$ implies $\dot{z}_{1}(t)=0$, which means that once the state $z_{1}(t)$ reaches $z_{1}(t)=0$, it will stay there rather than move away.
\end{remark}

\subsection{Lyapunov-based analysis}
Followed by Lemma \ref{lem1}, Lemma \ref{lem2} is provided as follows for Lyapunov-based analysis.
\begin{lemma}\label{lem2}
Consider the nonlinear system (\ref{nonlinear_system}). Assume that there exist a positive definite, radially unbounded and continuously differentiable
function $V:\mathbf{R}^{n}\rightarrow [0,\infty)$, constants $\omega$ and $\alpha\in(0,1)$ such that
\begin{align}
\dot{V}(x(t))|_{(\ref{nonlinear_system})}\leq & -\frac{\omega}{\alpha}  \varepsilon(t) V^{1-\alpha}(x(t)),\quad V(x(0))=V_{0}\label{dV_lem}\\
\dot{\varepsilon}(t) = & \omega  V^{\alpha}(x(t)),\quad \varepsilon(0)=\varepsilon_{0}\geq 0,\nonumber
\end{align}
then system (\ref{nonlinear_system}) is FxTS with the settling time $\pi/(2\omega)$. In addition, if
\begin{align}
\dot{V}(x(t))|_{(\ref{nonlinear_system})}= & -\frac{\omega}{\alpha}  \varepsilon(t) V^{1-\alpha}(x(t)),\quad V(x(0))=V_{0}\label{dV_lem_2}\\
\dot{\varepsilon}(t) = & \omega  V^{\alpha}(x(t)),\quad \varepsilon(0)=0,\nonumber
\end{align}
then system (\ref{nonlinear_system}) is EFxTS with the settling time $\pi/(2\omega)$.
\end{lemma}
\begin{proof}
Let us first focus on (\ref{dV_lem}). We will prove it by contradiction. Assume that $x(t)\neq 0,\forall t\in [0,\pi/(2\omega)]$. Denote $W(x(t))=V^{\alpha}(x(t))$. Then
\begin{align}
\left\{
\begin{aligned}
\dot{W}(x(t))\leq & -\omega \varepsilon(t),\\
\dot{\varepsilon}(t)= & \omega W(x(t)),
\end{aligned}
\right. \label{lem2_eq1}
\end{align}
for all $t\in[0,\pi/(2\omega)]$.

To facilitate the analysis, system (\ref{lem2_eq1}) is transformed into polar coordinates, where the amplitude and phase angle are defined as $
\rho(t)=\sqrt{W^2(x(t))+\varepsilon^2(t)}$ and $ \theta(t)=\mathrm{atan}2(\varepsilon(t),W(x(t)))$
respectively. Then we have
\[
W(x(t))=\rho(t)\cos(\theta(t)),\quad \varepsilon(t)=\rho(t)\sin(\theta(t)).
\]
By taking the time derivatives of $\rho$ and $\theta$, we know that
\begin{align*}
\dot{\rho}(t)= & \frac{W\dot{W}+\varepsilon\dot{\varepsilon}}{\sqrt{W^2 + \varepsilon^2}}\leq\frac{W(-\omega \varepsilon)+\varepsilon (\omega W)}{\sqrt{W^2 + \varepsilon^2}}=0,\\
\dot{\theta}(t) = & \frac{\dot{\varepsilon}W - \varepsilon\dot{W}}{W^2 + \varepsilon^2} \geq \frac{(\omega W)W - \varepsilon(-\omega \varepsilon)}{W^2 + \varepsilon^2} = \omega.
\end{align*}
It follows that
\begin{equation}
\theta\left(\frac{\pi}{2\omega}\right)-\theta(0)=\int_{0}^{\pi/(2\omega)}\dot{\theta}(s)\mathrm{d}s\geq \frac{\pi}{2}. \label{eq_0126_1}
\end{equation}
Notice that $\varepsilon_{0}\geq 0$ and $W(0)\geq 0$, which implies $\theta(0)\in [0,\pi/2]$ and $\theta(\pi/(2\omega))\in [\pi/2,\infty)$. Furthermore, due to the continuity of $\theta$ on $[0,\pi/(2\omega)]$, there exists $t_{1}\in [0,\pi/(2\omega)]$ such that $\theta(t_1)=\pi/2$. Then, we know that $W(t_{1})=0$, which implies that $x(t_{1})=0$, therefore contradicting the assumption. Thus, there exists $t_{1}\in[0,\pi/(2\omega)]$ such that $x(t_{1})=V(x(t_{1}))=0$. Besides, it is obtained from $\dot{V}(x)\leq 0,\forall t\geq 0$ and $V(x)\geq 0$ that $V(x(t))=0,\forall t\geq t_{1}$, which implies (\ref{nonlinear_system}) is FxTS with the settling time $\pi/(2\omega)$.

Besides, if (\ref{dV_lem_2}) holds, then it follows from $V(x(t))\geq 0$ and $\dot{V}(x(t))\leq 0$ that system (\ref{nonlinear_system}) is Lyapunov stable.  Then, if $x_{0}\neq 0$, we can deduce $x(t)=0,\forall t\geq \pi/(2\omega)$ and $x(t)\neq 0,\forall t\in [0,\pi/(2\omega))$ by following the same analytical method as Lemma \ref{lem1}. For the sake of brevity, the specific proof process is omitted.
\end{proof}

\subsection{Exact Prescribed-Time Control of Scalar Systems}
A simple example of a scalar system is presented in this subsection to demonstrate the applicability of Lemma \ref{lem2} in designing prescribed-time control laws.
\begin{theorem}\label{the1}
Let $T>0$ be a prescribed number. Then the closed-loop system consisting of the scalar system
\begin{equation}
\dot{x}(t)=u(t),\quad x(0)=x_{0},\quad x(t)\in \mathbf{R}, \quad u(t)\in \mathbf{R},\label{sca_1}
\end{equation}
and the prescribed time control law
\begin{align}
\left\{
\begin{aligned}
u(t)=&-\frac{\pi}{2^{1-\alpha/2}\alpha T}\mathrm{sig}^{1-\alpha}(x(t))\varepsilon(t),\quad \alpha\in (0,1), \\
\dot{\varepsilon}(t)= & \frac{\pi}{2^{1+\alpha/2}T}\vert x(t)\vert^{\alpha},\quad \varepsilon(0)=0,
\end{aligned}
\right.\label{sca_u}
\end{align}
is EFxTS with the settling time $T$.
\end{theorem}
\begin{proof}
Denote $V(x(t))=x^{2}(t)/2$, whose time-derivative can be calculated as
\begin{align}
\dot{V}(x(t))= & -\frac{\pi}{2^{1-\alpha/2}\alpha T}\vert x(t)\vert^{2-\alpha}\varepsilon(t)\nonumber\\
= & -\frac{\pi}{\alpha T}V^{1-\alpha/2}(x(t))\varepsilon(t)\nonumber\\
\triangleq & -\frac{\pi}{2\beta T}V^{1-\beta}(x(t))\varepsilon(t),\label{0306_1}
\end{align}
and
\[
\dot{\varepsilon}(t)=\frac{\pi}{2^{1+\alpha/2}T}\vert x(t)\vert^{\alpha}=\frac{\pi}{2T}V^{\alpha/2}(x(t))\triangleq \frac{\pi}{2T}V^{\beta}(x(t)).
\]
According to Lemma \ref{lem2}, we know that the closed-loop system consisting of (\ref{sca_1}) and (\ref{sca_u}) is EFxTS with the settling time $T$.

Next, we prove the boundedness of $\varepsilon(t)$ and $u(t)$. It is obtained from $\varepsilon(0)=0$ and $\dot{\varepsilon}(t)\geq 0$ that $\varepsilon(t)\geq 0,\forall t\in [0,\infty)$. Therefore, combining (\ref{0306_1}), we know that $\dot{V}(x(t))\leq 0$, which implies $V(x(t))\leq V_{0}$ and $x(t)\leq x(0)$. It follows from $V(x(t))=0,\forall t\geq T$ that
\[
\vert \varepsilon(t)\vert\leq \int_{0}^{T}\frac{\pi}{2T}V^{\beta}(x(s))\mathrm{d}s\leq \frac{\pi}{2}V^{\beta}_{0},\quad \forall t\in [0,\infty).
\]
Therefore, $\varepsilon(t)$ is bounded for all $t\in [0,\infty)$. Given the boundedness of both $x(t)$ and $\varepsilon(t)$, we know that $u(t)$ is bounded for all $t\in [0,\infty)$. Furthermore, since $x(t)=0,\forall t\geq T$, we have $u(t)=0,\forall t\geq T$. The proof is finished.
\end{proof}

\begin{remark}
In contrast to prescribed-time control approaches based on THF \cite{Song17auto}, the proposed control law (\ref{sca_u}) achieves precise prescribed-time control without exhibiting any infinite gains, thereby avoiding singularity problems.
\end{remark}

\begin{remark}
Compared with prescribed-time control approaches based on PDF \cite{Ding23tac}, the control law in (\ref{sca_u}) does not rely on any complicated time-varying functions, nor does it require any additional offline computations.
\end{remark}

For simulation, the prescribed time and the controller parameter are chosen as $T=1\mathrm{s}$ and $\alpha=0.5$,  respectively.  Fig. \ref{sca_V} shows the state $x(t)$ and the control input $u(t)$ for the
closed-loop system consisting of (\ref{sca_1}) and (\ref{sca_u}).

\begin{figure}[h]
\centering
\includegraphics[scale=0.9]{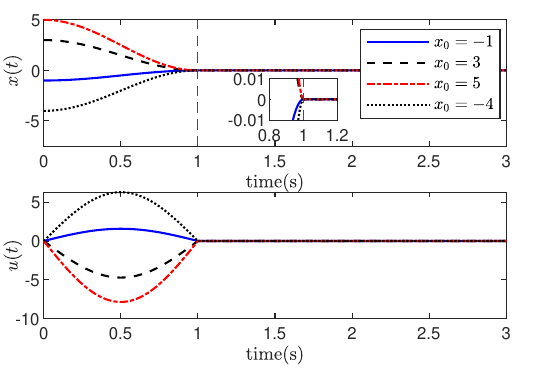}\caption{The state $x(t)$ and the control input $u(t)$ for (\ref{sca_1}) and (\ref{sca_u}).}%
\label{sca_V}%
\end{figure}

\section{Exact Prescribed Time Nonsingular Sliding Mode Control}\label{sec_sys}
\subsection{Design of Nonsingular Sliding Mode Control Law}

Consider the following nonlinear system
\begin{align}
\left\{
\begin{aligned}
\dot{x}_{1}(t)= & x_{2}(t),\\
\dot{x}_{2}(t)= &f(t,x(t))+g(t,x(t))u(t)+d(t,x(t)),
\end{aligned}
\right. \label{int2}
\end{align}
where $x(t)=[x_{1}(t),x_{2}(t)]^{\mathrm{T}}\in \mathbf{R}^{2}$ is the state, $x(0)=x_{0}$ is the initial value, $f(t,x(t))$ and $g(t,x(t))\neq 0$ are two known continuous nonlinear functions, and $d(t,x(t))$ represents the disturbance satisfying $\vert d(t,x(t))\vert\leq \bar{d}$ where $\bar{d}\geq 0$ is a known constant.

According to Lemma \ref{lem1}, a sliding mode variable for the system (\ref{int2}) can be designed as
\begin{align}
s(t)=x_2(t)+\frac{\pi}{2\alpha T_{\mathrm{s}}}\varepsilon_1 (t) \mathrm{sig}^{1-\alpha}(x_1(t)),\label{s}
\end{align}
where $\alpha \in (0,0.5)$ and $T_{\mathrm{s}}\in (0,\infty)$ are two constants,
\[
\dot{\varepsilon}_{1}(t)=\left\{
\begin{array}
[c]{ll}
0, & t\in [0,T_{\mathrm{c}}],\\
\frac{\pi}{2T_{s}}\vert x_{1}(t)\vert^{\alpha}, & t\in (T_{\mathrm{c}},\infty),
\end{array}
\right.
\]
and $\varepsilon_{1}(0)=0$.

Based on (\ref{s}), a nonsingular sliding mode control  law can be designed as
\begin{equation}
u(t)=\left\{
\begin{array}
[c]{ll}
g^{-1}(t,x(t))(u_{\mathrm{c}}(t)-f(t,x(t))), & t\in [0,T_{\mathrm{c}}],\\
g^{-1}(t,x(t))(u_{\mathrm{s}}(t)-f(t,x(t))), & t\in (T_{\mathrm{c}},\infty),
\end{array}
\right. \label{slide_u}
\end{equation}
where $T_{\mathrm{c}}\in (0,\infty)$ and
\begin{align*}
u_{\mathrm{c}}(t)= & -\frac{\pi}{2^{1-\alpha/2}\alpha T_{\mathrm{c}}}\mathrm{sig}^{1-\alpha}(x_{2}(t))\varepsilon_{2}(t)-\bar{d}\mathrm{sign}(x_{2}(t)),\\
u_{\mathrm{s}}(t)= & -\frac{\pi^2}{4\alpha T_{\mathrm{s}}^2} x_1(t) + \frac{\pi^2(1-\alpha)}{4\alpha^2 T_{\mathrm{s}}^2} \varepsilon_{1}^2(t) \mathrm{sig}^{1-2\alpha}(x_1(t))\\
&-\frac{\pi (1-\alpha)}{2\alpha T_{\mathrm{s}}}\varepsilon_{1}(t)\mathrm{sig}^{1-\alpha}(s(t))-\bar{d}\mathrm{sign}(s(t))\\
&-\frac{1+\alpha}{2}\vert x_2(t)\vert \mathrm{sign}(s(t)),
\end{align*}
in which
\[
\dot{\varepsilon}_{2}(t)= \frac{\pi}{2^{1+\alpha/2}T_{\mathrm{c}}}\vert x_{2}(t)\vert^{\alpha},\quad \varepsilon_{2}(0)=0,\quad t\in [0,\infty).
\]
\begin{theorem}\label{the2}
Let $T>0$ be a prescribed number, $T_{\mathrm{c}}\in (0,T)$ and $T_{\mathrm{s}}=T-T_{\mathrm{c}}$. The state of closed-loop system consisting of the system (\ref{int2}) and control law (\ref{slide_u}) satisfies $x(t)=0, \forall t\geq T$.
\end{theorem}
\begin{proof}
For $t\in [0,T_{\mathrm{c}}]$, it holds that
\begin{align}
\dot{x}_{2}(t)= & -\frac{\pi}{2^{1-\alpha/2}\alpha T_{\mathrm{c}}}\mathrm{sig}^{1-\alpha}(x_{2}(t))\varepsilon_{2}(t)\nonumber\\
&+d(t,x(t)) -\bar{d}\mathrm{sign}(x_{2}(t)).\label{0129_1}
\end{align}
Denote $V(x_{2}(t))=x_{2}^{2}(t)/2$, whose time-derivative can be calculated as
\begin{align*}
\dot{V}(x_{2}(t))= & -\frac{\pi}{2^{1-\alpha/2}\alpha T_{\mathrm{c}}} \vert x_{2}(t)\vert^{2-\alpha} \varepsilon_{2}(t)-\bar{d}\vert x_{2}(t)\vert+d(t,x(t))x_{2}(t)\nonumber\\
\leq & -\frac{\pi}{\alpha T_{\mathrm{c}}}V^{\frac{2-\alpha}{2}}(x_{2}(t))\varepsilon_{2}(t)\nonumber\\
\triangleq & -\frac{\pi}{2\beta T_{\mathrm{c}}}V^{1-\beta}(x_{2}(t))\varepsilon_{2}(t),
\end{align*}
and
\[
\dot{\varepsilon}_{2}(t)\!=\! \frac{\pi}{2^{1+\alpha/2}T_{\mathrm{c}}}\vert x_{2}(t)\vert^{\alpha}\!=\!\frac{\pi}{2T_{\mathrm{c}}}V^{\alpha/2}(x_{2}(t))
\!\triangleq\! \frac{\pi}{2T_{\mathrm{c}}}V^{\beta}(x_{2}(t)).
\]
According to Lemma \ref{lem2}, we know that $x_{2}(T_{\mathrm{c}})=0$.

For $t\in(T_{\mathrm{c}},\infty)$, it is obtained from $x_{2}(T_{\mathrm{c}})=0$ and $\varepsilon_{1}(T_{\mathrm{c}})=0$ that $s(T_{\mathrm{c}})=0$. Denote $
W(t)=\vert x_{1}(t)\vert^{1+\alpha}s^{2}(t)$, whose time-derivative can be calculated as
\begin{align}
\dot{W}(t)
=& \frac{\pi(1-\alpha)}{\alpha T_{\mathrm{s}}}\varepsilon_{1}(t)(\vert x_{1}(t)\vert s^2(t)-\vert x_{1}(t)\vert^{1+\alpha}\vert s(t)\vert^{2-\alpha})\nonumber\\
&-2\vert x_{1}(t)\vert^{1+\alpha}\bar{d}\vert s(t)\vert+2\vert x_{1}(t)\vert^{1+\alpha}d(t,x(t))s(t) \nonumber\\
&+(1+\alpha)\left(\mathrm{sig}^{\alpha}(x_{1}(t))x_{2}(t)s^2(t)-\vert x_{1}(t)\vert^{1+\alpha}\vert x_{2}(t)\vert \vert s(t)\vert\right)\nonumber\\
\leq & \frac{\pi(1-\alpha)}{\alpha T_{\mathrm{s}}}\varepsilon_{1}(t)(\vert x_{1}(t)\vert s^2(t)-\vert x_{1}(t)\vert^{1+\alpha}\vert s(t)\vert^{2-\alpha})\nonumber\\
&+(1+\alpha)\vert x_{2}(t)\vert(\vert x_{1}(t)\vert^{\alpha}s^2(t)-\vert x_{1}(t)\vert^{1+\alpha} \vert s(t)\vert).\label{d_Wt}
\end{align}
It can be seen that $\dot{W}(t)\leq 0$ when $\vert s(t)\vert\leq \vert x_{1}(t)\vert$.

Next, we proceed by considering two cases separately:

\textbf{The case $x_{1}(T_{\mathrm{c}})=0$: } It can be verified that $x(t)=0$ is the equilibrium of the closed-loop system. Thus $x(t)=0,\forall t\geq T_{\mathrm{c}}$ in this case.

\textbf{The case $x_{1}(T_{\mathrm{c}})\neq 0$: }
We will prove by contradiction. Suppose there exists $t_{1}\in [T_{\mathrm{c}},T]$ such that $\vert s(t_1)\vert>\vert x_1(t_1)\vert$. From $\vert s(T_{\mathrm{c}})\vert=0< \vert x_{1}(T_{\mathrm{c}}) \vert$ and the continuity of $s(t)$ and $x_{1}(t)$, it follows that there exists $t_{2}\in [T_{\mathrm{c}},t_{1})$ such that $
\vert s(t_{2})\vert=\vert x_{1}(t_{2})\vert, \vert s(t)\vert\leq\vert x_{1}(t)\vert, \forall t\in [T_{c},t_{2}]$.
According to (\ref{d_Wt}) and $W(T_{\mathrm{c}})=0$, we know that $W(t)=0,\forall t\in [T_{\mathrm{c}},t_2]$ and $s(t_{2})=x_{1}(t_{2})=0$. It follows from $x_{1}(T_{\mathrm{c}})\neq 0$ that there exists $t_{3}\in (T_{\mathrm{c}},t_{2}]$ such that $x_{1}(t)\neq 0,\forall t\in [T_{\mathrm{c}},t_{3})$ and $x_{1}(t_{3})=0$. Then, it is obtained from $W(t)=0,\forall t\in [T_{\mathrm{c}},t_{3})$ that $s(t)=0,\forall t\in [T_{\mathrm{c}},t_{3})$, namely,
\begin{align}
\left\{
\begin{aligned}
\dot{x}_{1}(t)= & x_{2}(t)=  -\frac{\pi}{2\alpha T_{\mathrm{s}}}\varepsilon_1 (t) \mathrm{sig}^{1-\alpha}(x_1(t)),\\
\dot{\varepsilon}_{1}(t)= & \frac{\pi}{2T_{s}}\vert x_{1}(t)\vert^{\alpha},\quad \varepsilon_{1}(T_{\mathrm{c}})=0,
\end{aligned}
\right. \label{s=0_0129}
\end{align}
for all $t\in [T_{\mathrm{c}},t_{3})$. We can conclude by using Lemma \ref{lem1} and $x_{1}(t)\neq 0, \forall t\in [T_{\mathrm{c}},t_{3})$ that $t_{3}\geq T$, which further implies $t_{1}>t_{3}\geq T$. It contradicts the assumption, thus $\vert s(t)\vert\leq \vert x_{1}(t)\vert,\forall t\in [T_{\mathrm{c}},T]$. By using (\ref{d_Wt}) and $W(T_{\mathrm{c}})=0$, we have
$W(t)=0,\forall t\in [T_{\mathrm{c}},T]$. It follows from $
\vert s(t)\vert^{3+\alpha}\leq W(t)=0, \forall t\in [T_{\mathrm{c}},T]$
that $s(t)=0,\forall t\in [T_{\mathrm{c}},T]$. Thus, (\ref{s=0_0129}) holds for all $t\in [T_{\mathrm{c}},T]$. According to Lemma \ref{lem1}, we know that $x_{1}(T)=x_{2}(T)=0$. It can be verified that $x(t)=0$ is the equilibrium of the closed-loop system, which implies $x(t)=0,\forall t\geq T$.

Moreover, similar to the proof of Theorem \ref{the1}, the boundedness of $u(t)$ and $\varepsilon_{i}(t),i=1,2$ and $u(t)=0,\forall t\geq T$ can be derived based on the boundedness of $x(t)$ and the fact $x(t)=0,\forall t\geq T$. For brevity, the details are omitted here. The proof is finished.
\end{proof}
\begin{remark}
Here, we briefly clarify that the settling time is independent of the initial conditions. From the proof of Theorem \ref{the2}, it follows that $x(t)=0,\forall t\geq T$  and the settling time is exactly $T$ unless $x_{1}(T_{\mathrm{c}})=0$. Besides, it is obtained from (\ref{0129_1}) and Lemma \ref{lem1} that $x_{2}(t)=x_{2}(0)\vert \cos(\omega t)\vert^{1/\alpha},\forall t\in [0,T_{\mathrm{c}}]$ when $d(t,x(t))=0$, which implies
\[
x_{1}(T_{\mathrm{c}})=x_{1}(0)+x_{2}(0)\int_{0}^{T_{\mathrm{c}}}\vert \cos(\omega s)\vert^{1/\alpha}\mathrm{d}s.
\]
Thus, unless in the extremely rare case where
\[
x_{1}(0)+x_{2}(0)\int_{0}^{T_{\mathrm{c}}}\vert \cos(\omega s)\vert^{1/\alpha}\mathrm{d}s=0,
\]
the settling time is exactly $T$. Therefore, the settling time can be regarded as almost completely independent of the initial conditions.
\end{remark}

\begin{remark}
It is worth noting that for $t\geq T_{\mathrm{c}}$, differentiating the sliding variable $s(t)$ defined by (\ref{s}) leads to the singularity problem similar to those encountered in terminal sliding mode control \cite{Feng02auto}. However, the control scheme proposed in this section circumvents the direct differentiation of $s(t)$ during both design and analysis. By employing $W(t)=\vert x_{1}(t)\vert^{1+\alpha}s^{2}(t)$ as a substitute, the singularity problem is effectively eliminated (it is apparent that the control law given by (\ref{slide_u}) is nonsingular).
\end{remark}

\begin{remark}
It should be noted that since $W(t)$ is not positive definite with respect to $s(t)$. Strictly speaking, it cannot serve as a Lyapunov function. However, as $s(t)$ has converged to zero by $t=T_{\mathrm{c}}$, it suffices to utilize $W(t)$ to demonstrate that $s(t)$ remains at zero without escaping. Thus, the role played by $W(t)$ is significantly less demanding than that of a typical Lyapunov function.
\end{remark}

\begin{remark}
Unlike existing nonsingular finite/fixed-time sliding mode methods (see, for example, \cite{Corradini18auto,Feng02auto,Moulay21tac,Rabiee19auto,Yang13auto} and \cite{Zuo15iet}), where the actual settling time typically remains dependent on the initial conditions, the method proposed in Theorem \ref{the2} yields an actual settling time that is almost independent of the initial conditions.
\end{remark}

\begin{remark}
Compared with the THF-based prescribed-time sliding mode control method \cite{Djennoune23ifac} and \cite{Shi22tcas}, the gain in control law (\ref{slide_u}) is uniformly bounded. Thus, no singularity issues occur at or after the prescribed time $T$.
\end{remark}

\begin{remark}
Compared with the PDF-based prescribed-time sliding mode control methods \cite{Ding26tac,Zhou24scis}, the method proposed in Theorem \ref{the2} does not rely on any periodic time-varying gains that require additional offline computations.
\end{remark}

\subsection{A Simulation Example}
The dynamics of the Van der Pol circuits system can be described as \cite{Rabiee19auto}
\begin{align}
\left\{
\begin{aligned}
\dot{x}_{1}(t)\!= & x_{2}(t),\\
\dot{x}_{2}(t)\!=\! & -\!2x_{1}(t)\!+\!3(1-x_{1}^{2}(t))x_{2}(t)+u(t)+d(t),
\end{aligned}
\right.\label{FDPL}
\end{align}
which can be rewritten as (\ref{int2}) with $
f(t,x(t))=-2x_{1}(t)+3(1-x_{1}^{2}(t))x_{2}(t), g(t,x(t))=1.$
For simulation, the external disturbance $d(t)$ is set to $
d(t)= 2\sin(0.2\pi t)+0.15\sin(2\pi t)$
and the initial value is set to $x(0)=[1,-1]^{\mathrm{T}}$.

\begin{figure}[h]
\centering
\includegraphics[scale=0.9]{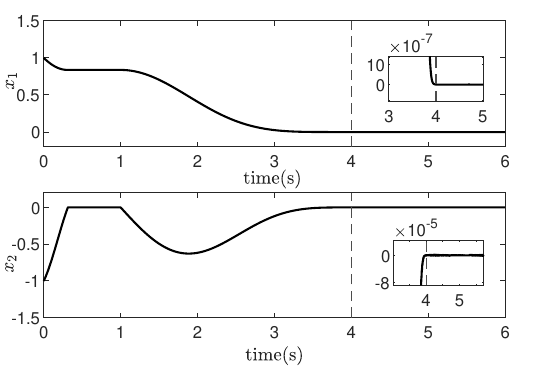}\caption{The state $x$ for system (\ref{FDPL}) and control law (\ref{slide_u}).}%
\label{smc_state}%
\end{figure}
\begin{figure}[h]
\centering
\includegraphics[scale=0.9]{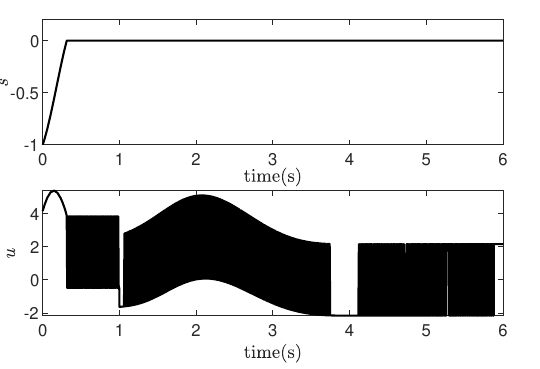}\caption{The sliding mode variable $s$ and the control input $u$
for system (\ref{FDPL}) and control law (\ref{slide_u}).}%
\label{smc}%
\end{figure}

According to Theorem \ref{the2}, the fixed-time sliding mode control law is designed as (\ref{slide_u}) with $
\alpha=0.2, \bar{d}=2.15, T=4, T_{\mathrm{c}}=1$.
Figs. \ref{smc_state} and \ref{smc} showcases the simulation results. Fig. \ref{smc} shows that the sliding mode variable can convergence to zero within $T_{\mathrm{c}}=1\mathrm{s}$. Fig. \ref{smc_state} shows that the state of the closed-loop system can convergence to zero at the prescribed time $T=4\mathrm{s}$.

\section{Conclusion}
This paper has addressed the problems of stability analysis and nonsingular sliding mode stabilization for exact prescribed-time control. By leveraging the isochronism property of simple harmonic motion, this paper developed a novel exact prescribed-time control framework that differs fundamentally from existing methods. First, a new Lyapunov analysis approach was established to characterize exact prescribed-time convergence, based on which exact prescribed-time stabilization for scalar systems was achieved. In particular, it was shown that for any nonzero initial condition, the closed-loop settling time matches the prescribed prescribed time exactly. Next, a novel exact prescribed-time nonsingular sliding mode control law was proposed. It ensures exact prescribed-time convergence for almost all initial conditions under disturbances, and the control law is well-defined throughout the entire state space with uniformly bounded gains. Simulation results demonstrated the effectiveness of the proposed methods.

\end{document}